\documentclass{amsart}
\usepackage{amsmath, amssymb}

\newtheorem{thm}{Theorem}
\newtheorem{lem}[thm]{Lemma}
\newtheorem{prop}[thm]{Proposition}

\newtheorem{rmk}[thm]{Remark}

\numberwithin{thm}{section}
\numberwithin{equation}{section}

\newcommand{\cv}{\mathbf{C}}
\newcommand{\zv}{\mathbf{Z}}
\newcommand{\bv}{\mathbf{B}}

\newcommand{\nv}{\mathbf{N}}
\newcommand{\aut}{\textup{Aut}}
\newcommand{\Jac}{\textup{Jac}}
\newcommand{\grad}{\textup{grad}}
\newcommand{\diag}{\textup{diag}}
\newcommand{\Diag}{\textup{Diag}}
\newcommand{\mbm}{\mathbf{m}}

\begin{document}

\title[The degree of biholomorphisms of quasi-Reinhardt domains]{The degree of biholomorphisms of quasi-Reinhardt domains fixing the origin}

\author{Feng Rong}

\address{School of Mathematical Sciences, Shanghai Jiao Tong University, 800 Dong Chuan Road, Shanghai, 200240, P.R. China}
\email{frong@sjtu.edu.cn}

\begin{abstract}
We give a description of biholomorphisms of quasi-Reinhardt domains fixing the origin via Bergman representative coordinates, which are shown to be polynomial mappings with a degree bound given by the so-called ``resonance order".
\end{abstract}

\keywords{quasi-Reinhardt domain; biholomorphism; resonance order; quasi-resonance order}

\subjclass[2010]{32A07, 32H02}

\thanks{The author is partially supported by the National Natural Science Foundation of China (Grant No. 11431008).}

\maketitle

\section{Introduction}

The study of ``special domains" invariant under a compact Lie group action is of classical interest (see e.g. \cite{H:special}), with Cartan's Linearity Theorem for circular domains being one of the most well-known results. While it is known that biholomorphisms between these special domains fixing the origin are all polynomials with uniform degree upper bound (see e.g. \cite{NZ:special}), such a uniform upper bound can be explicitly given in the case of quasi-circular domains and quasi-Reinhardt domains, thanks to the notion of \textit{resonance order} and \textit{quasi-resonance order} introduced in \cite{R:quasi, DR:auto}.

Let $T^r$ be the torus group of dimension $r\ge 1$. Let $\rho:T^r\rightarrow \textup{GL}(n,\cv)$ be a holomorphic linear action of $T^r$ on $\cv^n$ such that the only $\rho$-invariant holomorphic functions on $\cv^n$ are constant. Let $\lambda=(\lambda_1, \cdots, \lambda_r)\in T^r$, $\mbm_i=(m_{i1},\cdots,m_{ir})\in \zv^r$, $1\le i \le n$, and $z=(z_1,\cdots,z_n)\in\cv^n$. Then an action $\rho$ of $T^r$ on $\cv^n$ can be written as
\begin{equation}\label{E:rho}
\rho(\lambda)z=(\lambda^{\mbm_1}z_1,\cdots,\lambda^{\mbm_n}z_n),
\end{equation}
where $\lambda^{\mbm_i}=\prod_{j=1}^r \lambda_j^{m_{ij}}$ for each $i$. We call $\mbm_i$, $1\le i\le n$, the \textit{weight} of the action $\rho$.

Let $D$ be a bounded domain in $\cv^n$ containing the origin. We say that $D$ is \textit{quasi-Reinhardt} (of rank $r$) if it is $\rho$-invariant. When $r=1$, we say that $D$ is \textit{quasi-circular}.

Denote by $\nv$ the set of non-negative integers. Let $\alpha=(\alpha_1,\cdots,\alpha_n)\in \nv^n$ and $\mbm^j=(m_{1j},\cdots,m_{nj})$, $1\le j\le r$. For $1\le i\le n$, we define the \textit{i-th resonance set} as
$$E_i:=\{\alpha;\ \alpha\cdot \mbm^j=m_{ij},\ 1\le j\le r\},$$
where $\alpha\cdot \mbm^j=\sum_{i=1}^n \alpha_i m_{ij}$, and the \textit{i-th resonance order} as
$$\mu_i:=\max\{|\alpha|;\ \alpha\in E_i\},$$
where $|\alpha|:=\sum_{i=1}^n \alpha_i$. Then, the \textit{resonance order} is defined as
$$\mu:=\max\limits_{1\le i\le n} \mu_i.$$

Let $K_D(z,w)$ be the Bergman kernel, i.e. the reproducing kernel of the space of square integrable holomorphic functions on $D$. Since $D$ is bounded, we have $K_D(z,z)>0$ for $z\in D$. The Bergman metric tensor $T_D(z,w)$ is defined as the $n\times n$ matrix with entries $t_{ik}^D(z,w)=\frac{\partial^2}{\partial \bar{w}_i\partial z_k}\log K_D(z,w)$, $1\le i,k\le n$. For $z\in D$, we know that $T_D(z,z)$ is a positive definite Hermitian matrix (see e.g. \cite{B:Book}).

The Bergman representative coordinates at $\xi\in D$ is defined as (see e.g. \cite{Lu, GKK:Book})
\begin{equation}\label{E:B}
\sigma_\xi^D(z):=T_D(\xi,\xi)^{-1} \left.\grad_{\bar{w}}\log \frac{K_D(z,w)}{K_D(w,w)}\right|_{w=\xi}.
\end{equation}
As shown in \cite{LR:quasi}, one has $K_D(z,0)\equiv K_D(0,0)$ when $D$ is a quasi-Reinhardt domain. Thus the Bergman representative coordinates $\sigma_0^D(z)$ is defined for all $z\in D$.

For later use, we also record here the following well-known transformation formula for the Bergman metric tensor
\begin{equation}\label{E:T}
T_D(z,w)=\overline{\Jac_f^t(w)}T_D(f(z),f(w))\Jac_f(z),\ \ \ f\in \aut(D).
\end{equation}

The main purpose of this paper is to prove the following

\begin{thm}\label{T:main}
Let $f$ be a biholomorphism of quasi-Reinhardt domains $D_1$ and $D_2$, fixing the origin. Set $\sigma^i(z)=\sigma_0^{D_i}(z)$, $i=1,2$, and $J_f$ the linear part of $f$. Then,\\
\textup{(i)} $f=(\sigma^2)^{-1}\circ J_f\circ \sigma^1$;\\
\textup{(ii)} The degrees of $\sigma^i$ and $(\sigma^i)^{-1}$, $i=1,2$, are bounded by the resonance order of $D_i$.
\end{thm}

The proof of Theorem \ref{T:main} is given in section \ref{S:QR}. The corresponding result for automorphisms of quasi-circular domains fixing the origin was given by \cite[Lemma 3.1; Corollary 3.2; Lemma 3.3]{R:degree}, although the stronger statement \cite[Theorem 1.1]{R:degree} is not correct. We give a clarification of the situation in section \ref{S:QC}.

\section{Quasi-Reinhardt domains}\label{S:QR}

Let $D_1$ and $D_2$ be bounded quasi-Reinhardt domains containing the origin. Set $\sigma^i(z)=\sigma_0^{D_i}(z)$, $i=1,2$. Then, it is well-known (see e.g. \cite{GKK:Book}) that for any biholomorphism $f$ between $D_1$ and $D_2$ with $f(0)=0$ there exists a linear map $L_f$ between $\sigma^1(D_1)$ and $\sigma^2(D_2)$ such that
\begin{equation}\label{E:linear}
\sigma^2\circ f=L_f\circ\sigma^1.
\end{equation}

From the definition \eqref{E:B}, one readily checks that $\sigma^i(0)=0$ and
\begin{equation}\label{E:J}
\Jac_{\sigma^i}(z)=T_{D_i}(0,0)^{-1}T_{D_i}(z,0).
\end{equation}
In particular, one has $J_{\sigma^i}(z):=Jac_{\sigma^i}(0)\cdot z^t=z$, which also implies that $\sigma^i$'s are invertible. Combining these facts with \eqref{E:linear}, one sees that the linear map in \eqref{E:linear} is in fact just the linear part $J_f$ of $f$, and
\begin{equation}\label{E:linear2}
f=(\sigma^2)^{-1}\circ J_f\circ\sigma^1.
\end{equation}
This proves Theorem \ref{T:main}, (i).

To study the degree of the Bergman representative coordinates of a quasi-Reinhardt $D$, we first need to order the weight $\mbm_i$, $1\le i\le n$, in a proper way.

Without loss of generality and for simplicity, we will assume that all $\mbm_i$'s are distinct. (In the general case, whenever $\mbm_i=\mbm_j$ one can then treat $z_i$ and $z_j$ as the same, resulting in ``Jordan blocks", which do not affect the degree.)

We say that $\mbm_i<\mbm_j$ if there exists $\alpha\in E_j$ with $\alpha_i\neq 0$.

\begin{lem}\label{L:order}
For $1\le i\neq j\le n$, $\mbm_i<\mbm_j$ and $\mbm_j<\mbm_i$ can not hold at the same time.
\end{lem}
\begin{proof}
Assume that $\mbm_i<\mbm_j$ and $\mbm_j<\mbm_i$ hold at the same time. Then, there exists $\alpha\in \nv^n$ with $\alpha_i\ge 1$, $\alpha_j=0$ and $|\alpha|\ge 2$ such that
\begin{equation}\label{E:ij}
\alpha\cdot \mbm^k=m_{jk},\ \ \ \forall\ 1\le k\le r,
\end{equation}
and there exists $\beta\in \nv^n$ with $\beta_j\ge 1$, $\beta_i=0$ and $|\beta|\ge 2$ such that
\begin{equation}\label{E:ji}
\beta\cdot \mbm^k=m_{ik},\ \ \ \forall\ 1\le k\le r.
\end{equation}
Set $\gamma=(\gamma_1,\cdots,\gamma_n)$ with $\gamma_i=\beta_j\alpha_i-1$ and $\gamma_l=\beta_j\alpha_l+\beta_l$, $l\neq i$. Then one has $\gamma\in \nv^n$ and
$$|\gamma|=\beta_j|\alpha|+\sum_{l\neq i} \beta_l-1\ge |\beta|+(2\beta_j-1)\ge |\beta|+1\ge 3.$$
From \eqref{E:ij} and \eqref{E:ji}, one gets
$$\gamma\cdot \mbm^k=0,\ \ \ \forall\ 1\le k\le r.$$
This implies that $z^\gamma$ is invariant under the $\rho$-action \eqref{E:rho}, which contradicts with the assumption that the only $\rho$-invariant holomorphic functions on $\cv^n$ are constant.
\end{proof}

For a quasi-Reinhardt domain $D$ with weight $\mbm_i$, we say that $\mbm_i$'s are \textit{properly ordered} if $\mbm_i<\mbm_j$ only for $i<j$, $1\le i\neq j\le n$. (Note that such an ordering is only partial and not unique, and for definitiveness we can assign the larger indices to those $\mbm_i$'s without any resonance relations.)

A monomial $z^\alpha$ is called an \textit{i-th resonant monomial} if $\alpha\in E_i$.

\begin{prop}\label{P:resonant}
Let $D$ be a bounded quasi-Reinhardt domain containing the origin with weight $\mbm_i$ properly ordered. Set $\sigma(z)=\sigma_0^D(z)=(\sigma_1(z),\cdots,\sigma_n(z))$. Then for each $1\le i\le n$, $\sigma_i(z)=z_i+g_i(z)$, where $g_i(z)$ contains only nonlinear $i$-th resonant monomials. The same is true for $\sigma^{-1}(z)$.
\end{prop}
\begin{proof}
By \eqref{E:T}, we have
\begin{equation}\label{E:TD}
T_D(z,w)=\overline{\Jac_{\rho(\lambda)}^t(w)}T_D(\rho(\lambda)(z),\rho(\lambda)(w))\Jac_{\rho(\lambda)}(z).
\end{equation}
Setting $w=0$ in \eqref{E:TD}, we obtain
\begin{equation}\label{E:TD0}
T_D(z,0)=\diag(\lambda^{-\mbm_1},\cdots,\lambda^{-\mbm_n})T_D(\rho(\lambda)(z),0)\diag(\lambda^{\mbm_1},\cdots,\lambda^{\mbm_n}).
\end{equation}

Write $t_{ik}^D(z,0)=\sum\limits_{|\alpha|\ge 0} a_{ik}^\alpha z^\alpha=\sum\limits_{|\alpha|\ge 0} a_{ik}^\alpha z_1^{\alpha_1}\cdots z_n^{\alpha_n}$, $1\le i,k\le n$. Set $\beta=(\beta_1,\cdots,\beta_r):=(\alpha\cdot \mbm^1,\cdots, \alpha\cdot \mbm^r)$. Then from \eqref{E:TD0}, we get
\begin{equation}\label{E:a}
a_{ik}^\alpha=\lambda^{-\mbm_i+\mbm_k+\beta}a_{ik}^\alpha.
\end{equation}
Since \eqref{E:a} holds for any $\lambda\in T^r$, $a_{ik}^\alpha$ can be nonzero only when
$$-\mbm_i+\mbm_k+\beta=0,$$
which is satisfied if and only if
\begin{equation}\label{E:a1}
\alpha+e_k\in E_i.
\end{equation}
Here $e_k$ denotes the $k$-th unit multi-index.

Therefore, one can write $T_D(z,0)$ as
\begin{equation}\label{E:M}
T_D(z,0)=T_D(0,0)+M(z),
\end{equation}
where $T_D(0,0)=\Diag(\tau_1,\cdots,\tau_n)$ and $M(z)=[M_{ij}(z)]_{1\le i,j\le n}$ with $M_{ij}=0$ for $i\le j$.

By \eqref{E:M} and \eqref{E:J}, one has
\begin{equation}\label{E:Jac}
\Jac_\sigma(z)=T_D(0,0)^{-1}T_D(z,0)=I_n+T_D(0,0)^{-1}M(z)=:I_n+N(z),
\end{equation}
where $N(z)=[N_{ij}(z)]_{1\le i,j\le n}$ with $N_{ij}(z)=0$ for $i\le j$ and $N_{ij}(z)=\tau_i^{-1}M_{ij}(z)$ for $i>j$.

Since $\sigma(0)=0$, from \eqref{E:Jac} and \eqref{E:a1}, one sees that $\sigma(z)$ is of the desired form.

Since $J_\sigma=Id$ and each $g_i(z)$ only contains terms involving $z_j$'s with $j<i$ and of weighted degree equal to $\mbm_i$, a routine induction shows that each component of $\sigma^{-1}(z)$ also only contains resonant monomials (cf. \cite[Corollary 3.2]{R:degree}).
\end{proof}

Obviously, Proposition \ref{P:resonant} implies Theorem \ref{T:main}, (ii).

\begin{rmk}
Theorem \ref{T:main} gives a complete description of all possible forms of biholomorphisms of quasi-Reinhardt domains fixing the origin. It also gives a more transparent description of the ``quasi-resonance order" (cf. \cite{R:quasi, DR:auto}).
\end{rmk}

\section{Quasi-circular domains}\label{S:QC}

Let $D$ be a bounded quasi-circular domain containing the origin. The weight in the quasi-circular case is given by a set of $n$ positive integers $m_i$, $1\le i\le n$, with $\textup{gcd}(m_1,\cdots,m_n)=1$. A standard proper ordering of $m_i$'s is requiring that $m_i\le m_j$ for $i<j$.

First of all, the weight of a quasi-circular domain $D$ is not unique, if $D$ is in fact a quasi-Reinhardt domain of rank greater than one. For instance for $D=\bv^2$, the unit ball in $\cv^2$, any $(m_1,m_2)\in (\zv^+)^2$ is a weight. The classical Cartan's Linearity Theorem applies to $\bv^2$ and says that an automorphism of $\bv^2$ fixing the origin is linear. From the point of view of the resonance order, the weight which dictates the linearity is $(1,1)$ or $(m_1,m_2)$ with $\textup{gcd}(m_1,m_2)=1$. Therefore, if one considers $\bv^2$ as a quasi-circular domain with weight $(1,m)$ with $m\ge 2$, then one can not get the desired information on the degree of its automorphisms.

Secondly, the weight of a quasi-circular domain is not a biholomorphic invariant in general. For instance, if one considers $D=\phi_k(\bv^2)$ with $\phi_k(z_1,z_2)=(z_1,z_2+z_1^k)$, then $D$ is quasi-circular with weight $(1,k)$, but as we just noted above $\bv^2$ can be considered as a quasi-circular domain with any weight. And as noted in \cite[Example 5.1]{DR:auto}, the rank of a quasi-Reinhardt domain is also not invariant under biholomorphisms, and one can consider the \textit{maximal rank} of a quasi-Reinhardt domain.

We can define a \textit{genuine} quasi-circular domain to be a quasi-Reinhardt domain with maximal rank one. For the study of the degree of automorphisms and biholomorphisms of quasi-circular domains, the genuine case and the ``fake" case are vastly different. In dimension two, a complete study of the degree of origin-preserving automorphisms of quasi-circular domains was carried out in \cite{YZ:quasi}. In higher dimensions, a similar study would be very complicated, especially for ``fake" quasi-circular domains.

\end{document}